\documentclass[12pt]{article}
\usepackage{amsmath,latexsym,amssymb}
\usepackage{enumerate}
\usepackage{amsthm}
\usepackage{pgfplots}

\usepgfplotslibrary{fillbetween,decorations.softclip}
\pgfplotsset{compat=1.14}

\DeclareMathOperator\sP{P}   

\newtheorem{theorem}{Theorem}[section]

\newtheorem{example}[theorem]{Example}
\newtheorem{lemma}[theorem]{Lemma}
\newtheorem{remark}[theorem]{Remark}
\newtheorem{proposition}[theorem]{Proposition}

\newcommand{\address}{Address: Facultad de Ciencias F\'isico-Matem\'aticas, Universidad Aut\'onoma de Sinaloa, Culiac\'an, Sinaloa 80010, M\'exico; E-mail: islas.jose@uas.edu.mx}

\title{A sharp bound for winning within a proportion of the maximum of a sequence}
\author{Jos\'e A. Islas\footnote{\address}}

\begin{document}

\maketitle

\begin{abstract}
This note considers a variation of the full-information secretary problem where the random variables to be observed are independent and  identically distributed. Consider $X_1,\dots,X_n$ to be an independent sequence of random variables, let $M_n:=\max\{X_1,\dots,X_n\}$, and the objective is to select the maximum of the sequence. What is the maximum probability of ``stopping at the maximum''? That is, what is the stopping time $\tau$ adapted to $X_1,...,X_n$ that maximizes $P(X_{\tau}=M_n)$? This problem was examined by Gilbert and Mosteller \cite{GilMost} when in addition the common distribution is continuous. The optimal win probability in this case is denoted by $v_{n,max}^*$. What if it is desired to ``stop within a proportion of the maximum''? That is, for $0<\alpha<1$, what is the stopping rule $\tau$ that maximizes $P(X_{\tau} \geq \alpha M_n)$? 
In this note both problems are treated as games, it is proven that for any continuous random variable $X$, if $\tau^*$ is the optimal stopping rule then $P(X_{\tau^*} \geq \alpha M_n)\geq v_{n,max}^*$, and this lower bound is sharp. Some examples and another interesting result are presented.

\end{abstract}

\section{Introduction}

In the classical secretary problem (CSP) or the largest of an unknown set of numbers (Googol), a manager must select the best applicant among a known number $n$ of applicants. There is only one position available and he will interview one by one in random order. Immediately after the interview he must decide whether or not the applicant is hired based only on the relative ranks of the applicants interviewed so far. An applicant once rejected cannot be recalled. It is well known that using the optimal strategy, the probability of success decreases monotonically to $1/e\approx .3679$ as $n\to\infty$. For an entertaining exposition of the CSP see Ferguson \cite{Ferguson}.\\

An immediate extension of the CSP is the full-information best-choice problem in which each applicant independent of the others, can receive a numerical score whose distribution is known in advance and the objective is to select the highest score. The mathematical framework for this problem is as follows. Let $X_1,\dots,X_n$ be independent random variables with known continuous distributions, let $M_n:=\max\{X_1,\dots,X_n\}$, and consider $\sup_\tau \sP(X_\tau=M_n)$ where the supremum is over all stopping times adapted to the natural filtration of $X_1,\dots,X_n$. Gilbert and Mosteller \cite{GilMost} solved this optimal stopping problem when the sequence is independent and identically distributed. The optimal stopping rule depends on a sequence of decision numbers $d_i$, $1\leq i \leq n$, so that a gambler must stop at the first observation $i$ such that $X_i=M_i$ and $F(X_i)\geq d_i$, or at time $n$ if no such $i$ exists. They computed the optimal win probability for some values of $n$, denoted by $v_{n,max}^*$, are shown in Table $1$. This optimal win probability is independent of the distribution of the $X_i$'s. What if it is desired to ``stop within a proportion of the maximum''? That is, for $0<\alpha<1$, what is the stopping rule $\tau$ that maximizes $P(X_{\tau} \geq \alpha M_n)$? The examples in Section $2$ will show that the optimal stopping rule is cumbersome to calculate and that the win probability depends on the distribution. Moreover there is no close form expression to compute the optimal win probability as examined in \cite{GilMost}. In this note these two problem are treated as games for comparison. Notation and definitions are now introduced.\\

\noindent \textbf{Game Max.} Win is to stop at an observation that is the maximum of the sequence; 
\[ X_{\tau} \geq \max\{X_1,\dots,X_n\}.\]
\textbf{Game Proportion of the Max.} Win is to stop at an observation that is at least a proportion $0<\alpha<1$ of the maximum of the sequence; 
\[ X_{\tau} \geq \alpha \max\{X_1,\dots,X_n\}.\]
Here $\tau$ is any stopping rule adapted to $X_1,\dots,X_n.$\\

For the i.i.d. sequence $(X_1,\dots,X_n)$ of random variables having the same distribution as $X$, let
\begin{itemize}
\item $M_i:=\max\{X_1,\dots,X_i\}$, $i=1,2,\dots,n$,
\item $V_{n,max}(X,\tau)$ be the probability of win playing the Game Max for $X$ using the stopping rule $\tau$,
\item $V_{n,\alpha max}(X,\sigma)$ be the probability of win playing the Game Proportion of the Max for $X$ using the stopping rule $\sigma$,
\item $\tau^*$ be the optimal stopping rule for Game Max,
\item $\sigma^*$ be the optimal stopping rule for Game Proportion of the Max,
\item $V_{n,max}^*(X):= \displaystyle \sup_{\tau} V_{n,max}(X,\tau)=V_{n,max}(X,\tau^*)$,
\item $V_{n,\alpha max}^*(X):= \displaystyle \sup_{\sigma} V_{n,\alpha max}(X,\sigma)=V_{n,\alpha max}(X,\sigma^*)$,
\end{itemize}
A candidate is defined to be an observation that has a positive probability to win the game in question. In Game Max $X_i$ is a candidate if $X_i\geq \alpha M_i$. Thus, suppose a gambler is to observe $X_1,\cdots, X_n$, he can win at observation $i$ if $X_i$ is a candidate and the remaining $n-i$ are no greater than $X_i$. Similarly in Game Proportion of the Max, $X_i$ is a candidate if $X_i\geq \alpha M_i$. Note that if a gambler is playing both games when he wins in Game Max, he wins in Game Proportion of the Maximum as well. The converse is not necessarily true, this makes Game Proportion of the Max "easier" to win and hence the optimal win probability in Game Proportion of the Max is at least the optimal win probability in Game Max. The main result of this note makes this precise.  
\begin{theorem}\label{eq:thmChp3}
For each $n\geq 1$, for any continuous random variable $X$, \[V_{n,\alpha max}^*(X) \geq v_{n,max}^*\] and the inequality is sharp.
\end{theorem}
Where the values for $v_{n,max}^*$ were found by  \cite{GilMost} and are given in Table $1$.\\

\begin{remark}
{\rm
Although the Theorem is stated when the common distribution of $X$ is continuous, allowing the distribution to be discrete should make the Game Proportion of the Max easier to win due to the ties. Hence the sharp lower bound should hold as well. It is not clear how to make a proof of this. 
}
\end{remark}

The idea of the proof is to examine the worst case for a gambler. This is, playing Game Proportion of the Max with a distribution space out enough so that the advantage over Game Max is lost. The adequate distribution is given in at the beginning of Section $3$. Then in Lemma $3.1$ it is proven that the win probability of Game Proportion of the Max is arbitrarily close to Game Max. Finally, using the respective optimal stopping rules from both games and the use of inequalities, Theorem $1.1$ yields. 

\begin{table}[h]
\centering
\caption{Optimal win probabilities for some values of $n$ observations in Game Max \cite{GilMost}}
\begin{tabular}{l|l|l|l} \hline
n & $v_{n,max}^*$  & n        & $v_{n,max}^*$  \\ \hline \hline
1 & 1.0000  & 10       & .608699 \\
2 & .750000 & 15       & .598980 \\
3 & .684293 & 20       & .594200 \\
4 & .655396 & 30       & .589472 \\
5 & .639194 & 40       & .587126 \\
  &         & 50       & .585725 \\
  &         & $\infty$   & .580164\\ \hline
\end{tabular}
\end{table}

\section{Examples}
The following distributions are examined: discrete uniform and uniform. The following win probabilities are needed for computations, for $k=1,\dots,n$ let
\begin{align*}
U_k(x_1,\dots,x_k):&=P\left(x_k \geq \alpha M_n|X_1=x_1,\dots,X_k=x_k\right)\\
&=\begin{cases}
P\left(\max\{X_{k+1},\dots,X_n\} \leq \frac{ x_k}{ \alpha}\right) & \mbox{if } x_k\geq \alpha \max\{x_1,\dots,x_{k-1}\}\\
0, & \mbox{otherwise},
\end{cases}
\end{align*}
where $\max \emptyset:= -\infty$. Denote the expression in the first case above simply by $U_k(x_k)$, and note that
\[ U_k(x_k)=\left(F\left(\frac{x_k}{\alpha}\right)\right)^{n-k}. \]
Let
\[W_k(x_1,\dots,x_k):= \sup_{k < \tau \leq n} P(X_{\tau}\geq \alpha M_n|X_1=x_1,\dots,X_k=x_k).\]
$U_k(x_1,\dots,x_k)$ is interpreted as the win probability if stopping at time $k$ after observing $X_1=x_1,\dots,X_k=x_k$, and $W_k(x_1,\dots,x_k)$ as the optimal win probability when taking at least one more observation after observing  $X_1=x_1,\dots,X_k=x_k$.\\

\begin{example}Discrete Uniform Distribution. Let
\begin{center}
$P(X=x)=p(x)=1/10 \mbox{ for } x=1,2,\dots,10.$
\end{center}
Case $n=2$. The random variables $X_1, X_2$ will be observed and the optimal stopping rule is found using backward induction.\\
At observation $2$ we must stop. At observation $1$ we stop if  
\begin{align*}
U_1(x_1)=P\left(x_1 \geq \alpha X_2 \right)&\geq W_1(x_1)=P\left(X_2 \geq \alpha x_1 \right) &\Leftrightarrow \\  
\left\lfloor{ \frac{x_1}{\alpha} }\right\rfloor +  \left\lceil{ \alpha x_1 }\right\rceil  & \geq  11   
\end{align*}
Thus $x_1^*(\alpha)$ must be found so that for $x_1\geq x_1^*(\alpha)$ the previous inequality holds.

If $\alpha\geq 1/2$ and we observe a value less than $5$, continuing give us a higher probability of win. The probability of win is\\
\begin{align*}
E\left(\max\{W_1(X),U_1(X)\}\right)&=\sum_{x=1}^{10} \max \left\{W_1(x),U_1(x)\right\}p(x) \\
&= \frac{1}{10}\left(\sum_{x=x_1^*(\alpha)}^{10} \min\left\lbrace 1, \frac{  \left\lfloor{ \frac{x_1}{\alpha} }\right\rfloor}{10} \right\rbrace  + \sum_{x=1}^{x_1^*(\alpha)-1}  \left( \frac{ 11 - ( \left\lceil{ \alpha x_1 }\right\rceil )}{10} \right) \right) \\
\end{align*}
For the given values of $\alpha$, Table \ref{tab:discrete-uniform-win-prob} below shows the corresponding $x_1^*(\alpha)$ and the respective win probability. The table suggests that for small values of $\alpha$ the Game Proportion of the Max can always be won. The following proposition generalizes this idea.

\begin{table}[h]
\centering
\caption{Optimal win probabilities for discrete uniform distribution on $\{1,2,\dots,10\}$}
\label{my-label}
\begin{tabular}{l||l|l|l|l|l|l|l|l|l}
\hline
$\alpha$                       & 0.1 & 0.2 & 0.3 & 0.4 & 0.5 & 0.6 & 0.7 & 0.8 & 0.9 \\ \hline
$x_1^*(\alpha)$ & 1   & 2   & 3   & 4   & 5   & 5   & 5   & 5   & 6   \\ \hline
$P(win)$                       &  1   &  1   & 1    &  0.99   & 0.98    & 0.94    & 0.9    & 0.86    &  0.81 \\ \hline
\end{tabular}
\label{tab:discrete-uniform-win-prob}
\end{table}
\end{example}

\begin{proposition}
For any $n\geq 2$, let $X$ be a random variable. Then $V_{n,\alpha max}^*(X)=1$ if and only if $X$ has support on $[m,M]$ where $0<m\leq M<\infty$ and either (i) $\alpha^2 \leq \frac{m}{M}$ or (ii) $P(\frac{m}{\alpha}<X<\alpha M)=0$. 
\end{proposition}
\begin{proof}
First, we show that if (i) or (ii) then  $V_{n,\alpha max}^*(X)=1$. Consider $\tau:= \min\{1 \leq i \leq n-1:  X_i \geq \alpha M \}$ or $\tau:=n$ if no such $i$ exists. 
Under either of the two conditions, it is trivial that for each $1\leq i \leq  n-1$, \[P(X_i \geq \alpha \max\{X_{i+1},\dots,X_n\}| \max\{X_1,\dots,X_{i-1}\}< \alpha M, X_i \geq \alpha M)=1.\]
And at observation $n$, if (i), we have 
\[P(X_n \geq \alpha \max\{X_1,\dots,X_{n-1}\} | \max\{X_1,\dots,X_{n-1}\} < \alpha M)\geq P(X_n \geq m)=1.\]
If (ii), we have   
 \[P(X_n \geq \alpha \max\{X_1,\dots,X_{n-1}\} | \max\{X_1,\dots,X_{n-1}\} < \alpha M)\geq P(X_n \geq m)=1,\]
since $\max\{X_1,\dots,X_{n-1}\} < \alpha M$ and $X$ does not have mass between $\frac{m}{\alpha}$ and $\alpha M$ it follows that $\max\{X_1,\dots,X_{n-1}\} \leq \frac{m}{\alpha}$, thus $\alpha \max\{X_1,\dots,X_{n-1}\} \leq m$. Hence for any of the two conditions $V_{n,\alpha max}^*(X)=1$. Conversely,  it is clear that the distribution of $X$ must have support on $[m,M]$ for some $0<m\leq M<\infty$ and either $\alpha^2 \leq \frac{m}{M}$ or $\alpha^2 > \frac{m}{M}$ is satisfied. Suppose $\alpha^2 > \frac{m}{M}$ and assume that for  some $\epsilon>0$, $X$ has mass on the interval $I:=[\frac{m}{\alpha}+\epsilon,\alpha M - \epsilon]$. If we observe $X_1 \in I$ and stop, we lose if there exists $l>1$ such that $M-\frac{\epsilon}{\alpha} < X_l$. If we continue, we lose if for all $l>1$, $X_l< m + \alpha \epsilon$. It follows that no stopping rule gives a win probability of $1$ unless the interval $I$ has probability $0$. Hence (ii) must be satisfied. 
\end{proof}

\begin{remark}
{\rm
This result holds for any random variable while the Theorem is proven for a continuous random variable.
}
\end{remark}

\begin{example}Continuous Uniform Distribution on $(0,1)$\\
Case $n=2$. The random variables $X_1, X_2$ will be observed and the optimal stopping rule is found using backward induction.\\
At observation $2$ we must stop. At observation $1$ stop if 
\begin{align}
U_1(x_1)=P\left(x_1 \geq \alpha X_2 \right)&\geq W_1(x_1)=P\left(X_2 \geq \alpha x_1 \right) &\Leftrightarrow  \nonumber \\  
\min\{1,\frac{x_1}{\alpha}\} &\geq 1 - \alpha x_1   &\Leftrightarrow \nonumber \\ 
x_1 &\geq \frac{\alpha}{1+\alpha^2}:=x_1^*(\alpha). \nonumber
\end{align}
The probability of win is
\begin{align*}
E\left(\max\{W_1(X),U_1(X)\}\right)&= \int_{x_1^*(\alpha)}^1 U_1(x)dx + \int_0^{x_1^*(\alpha)} W_1(x)dx \\
&= \int_{x_1^*(\alpha)}^\alpha \frac{x}{\alpha}dx + \int_{\alpha}^1 1 dx + \int_0^{x_1^*(\alpha)} \left(1-\alpha x \right)dx \\
&=1-\frac{\alpha^3}{2(\alpha^2+1)}.
\end{align*}
This is graphed in Figure \ref{fig:uniform-n2-win-prob}.

\begin{figure}[h]
\centering
\caption{Optimal win probabilities for the uniform$(0,1)$ distribution, $n=2$} \label{fig:wp_1}
\begin{tikzpicture}
  \pgfdeclarelayer{pre main}
  \pgfsetlayers{pre main,main}
  \begin{axis}[
      axis lines = middle,
      enlargelimits,
      domain  = 0:1,
			xtick={0.5,1},
			ytick={0.75,0.9,1},
      xlabel  = {$\alpha$},
      ylabel  = {$P(win)$},
			xlabel style={below right},
      ylabel style={above left},
      xmin    = 0,
      xmax    = 1,
      ymin    = 0.5,
      ymax    = 1,
          ]

 		\addplot [name path=l3 ]  {1-x/2+ (x/(1+x*x))*( -1/(2*(1+x*x))+1-(x*x/(2*(1+x*x))) ) };

		\end{axis}	
\end{tikzpicture}
\label{fig:uniform-n2-win-prob}
\end{figure}
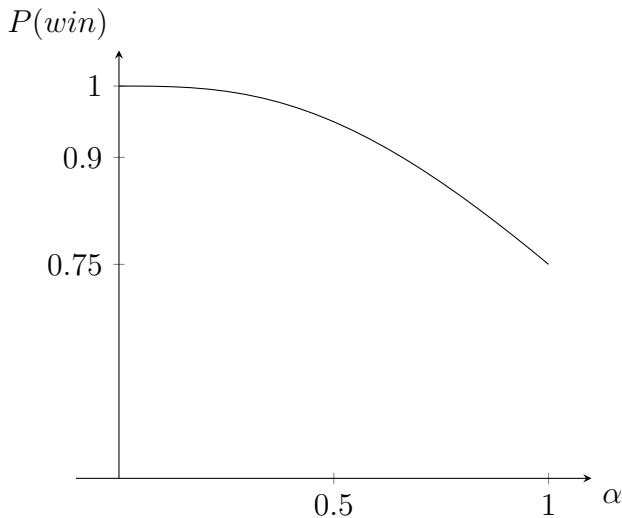

\end{example}

From the previous examples it can been seen that the optimal stopping rule is cumbersome to calculate since it depends on the common distribution. The probability of winning is in general high but, how low can the optimal win probability be in Game Proportion of the Max? The main result of this note provides the best lower bound.

\begin{section}{ A sharp lower bound.}
To prove Theorem $1.1$, a family of random variables is introduced with the property that the values are so spread out so that advantage in Game Proportion of the Max over Game Max is lost. In other words, the only possibility to win in Game Proportion of the Max is to choose the maximum of the sequence. 

Let $N_{\alpha}=\min\{N \in \mathbb{N}| \alpha > \frac{1}{N}\}$ and consider the random variable $X^{k,\epsilon}$ with density 
\[
f_{X^{k,\epsilon}}(x) =
  \begin{cases}
    \frac{1}{2\epsilon k}  & \text{if } (N_{\alpha}+1)^j-\epsilon \leq x \leq (N_{\alpha}+1)^j+\epsilon \text{ for } j=1,\dots,k,\\
    0  & \text{otherwise}, \\	
  \end{cases}
\]
$k \in \mathbb{N}$. The value of $\epsilon$ must be chosen such that for $j=1,\dots,k-1$, the inequality \[
(N_{\alpha}+1)^j+\epsilon < \alpha (N_{\alpha}+1)^{j+1} - \epsilon \]
holds. Solving, it is enough for $\epsilon$ to satisfy
\[
\frac{(N_{\alpha}+1)^j(\alpha N_{\alpha}+\alpha -1)}{\alpha+1}>\epsilon
\]
when $j=1$. The intervals where this density is positive are enough space out so that the gambler's advantage playing the Game Proportion of the Max over the Game Max is lost. The following lemma makes this precise.
\begin{lemma} \label{eq:lemma4}
For every $n\geq 1$, given $\delta>0$, there exists $k>1$ such that\[  
V_{n,\alpha max}(X^{k,\epsilon},\tau) \leq V_{n,max}(X^{k,\epsilon},\tau) + \delta\]
for any stopping rule $\tau$ adapted to $X_1,\cdots, X_n$.
\end{lemma}
\begin{proof}
When $n=1$, the result is trivial. Let $n>1$ and $\delta >0$. The idea of the proof is that a gambler is observing the sequence and playing both games with the same strategy. It's straightforward that if he wins in the Game Max, he also wins in the Game Proportion of the Max but the converse may not be true. \\
Thus, 
\begin{equation}\label{eq:ineqlemma}
V_{n,\alpha max}(X^{k,\epsilon},\tau) - V_{n,max}(X^{k,\epsilon},\tau)= P(B),
\end{equation}
where $B$ is the event that the gambler wins in Game Proportion of the Max and loses in Game Max using the stopping rule $\tau$. To simplify calculations, the following random variables are introduced. For $1 \leq i \leq n$, let \[
Z_i =
  \begin{cases}
    j  & \text{if } (N_{\alpha}+1)^j-\epsilon \leq X_i^{k,\epsilon}  \leq (N_{\alpha}+1)^j+\epsilon \text{ for } j=1,\dots,k,\\
  \end{cases}
\]
then the event $B$ is contained in the event where the sequence ($Z_1,\dots,Z_n$) does not have a unique maximum. It follows that 
\begin{align*}
P(B) &\leq P((Z_1,\dots,Z_n) \text{ does not have a unique maximum})\\
&=1-P((Z_1,\dots,Z_n) \text{ has a unique maximum}).
\end{align*}
For $i=1,\dots,n$, let $A_i$ be the event that the sequence ($Z_1,\dots,Z_n$) has the unique maximum at observation $i$. These events are clearly mutually exclusive. Then
\begin{align*}
P(B)&\leq 1- P\left( \bigcup_{i=1}^n A_i \right)\\
&=1- \sum_{i=1}^n P\left( A_i \right).
\end{align*}
Since the unique maximum $Z_i=j$ for some $j \in \{2,\dots,k \}$, each $A_i$ can be expressed as the union of mutually exclusive events
\[A_i= \bigcup_{j=2}^k \left[ \left\{ Z_i=j \right\}\cap \bigcap_{l\neq i} \left\{ Z_l<Z_i \right\} \right].\]
Thus 
\[ P\left( A_i \right)= \sum_{j=2}^k P\left( \left\{ Z_i=j \right\}\cap \bigcap_{l\neq i} \left\{Z_l<j \right\}\right), \]
and so
\begin{align*}
P(B)& \leq 1- \sum_{i=1}^n \sum_{j=2}^k P\left( \left\{ Z_i=j \right\}\cap \bigcap_{l\neq i} \left\{Z_l<j \right\}\right)\\
&=1- n\sum_{j=2}^k \frac{1}{k} \left(\frac{j-1}{k} \right)^{n-1}.
\end{align*}
Now, the limit of the summation as  $k \rightarrow \infty$ is found: 
\[ \lim_{k \to \infty} \sum_{j=2}^{k} \frac{1}{k}\left(\frac{j-1}{k}\right)^{n-1}=\lim_{k \to \infty} \sum_{j=0}^{k-1} \frac{1}{k}\left(\frac{j}{k}\right)^{n-1}
	=\int_0^{1}t^{n-1}dt = \frac{1}{n},\]
since the summation on the left hand side is a Riemann sum of $f(t)=t^{n-1}$ over $I=[0,1]$.\\
So, let 
\[k_{\delta}:= \min\left\{k\geq 1:\frac{1-\delta}{n} <  \sum_{j=1}^{k-1} \frac{1}{k}\left(\frac{j}{k}\right)^{n-1}\right\}. \]
Then, 
\begin{align*}
P(B) &\leq 1-n\sum_{j=1}^{k_{\delta}-1} \frac{1}{k_{\delta}} \left(\frac{j-1}{k_{\delta}} \right)^{n-1}\\
&<1-n\frac{1-\delta}{n} \\
&=\delta.
\end{align*}
Thus \eqref{eq:ineqlemma} holds for $X^{k_{\delta},\epsilon}$. This completes the proof.
\end{proof}

Now, using the previous Lemma the proof of the Theorem is simple.
\begin{proof}[Proof of Theorem \ref{eq:thmChp3}] 
Let $n\geq1$. It is straightforward that for any $\tau$
\begin{equation*}
\{X_\tau\geq \alpha M_n \} \supseteq \{X_\tau\geq M_n\}.  
\label{eq:ineq}
 \end{equation*}
(If the gambler wins in Game Max, he also wins in Game Proportion of the Max), thus 
\[V_{n,\alpha max}(X,\tau) \geq V_{n,max}(X,\tau).  \]
In particular, when using the optimal stopping rule $\tau^*$ of the Game Max in both games, it follows that 
\[V_{n,\alpha max}^*(X) \geq V_{n,\alpha max}(X,\tau^*) \geq V_{n,max}(X,\tau^*) = v_{n,max}^* \]
Now to prove that the bound is sharp, let $\delta >0$. By Lemma \ref{eq:lemma4} there exists $k_{\delta}$ such that for the $X^{k_{\delta},\epsilon}$ distribution if the gambler uses the optimal stopping rule $\sigma^*$ for the Game Proportion of the Max in both games,  
\begin{align*}
V_{n,\alpha max}^*(X^{k_{\delta},\epsilon})= V_{n,\alpha max}(X^{k_{\delta}, \epsilon},\sigma^*) &\leq V_{n,max}(X^{k_{\delta},\epsilon},\sigma^*) + \delta \\
&\leq V_{n,max}^*(X^{k_{\delta},\epsilon})+\delta\\
&= v^*_{n,max} + \delta.
\end{align*}
Thus for any $\delta>0$ there exists a distribution such that the optimal win probability in Game Proportion of the Max is at most $v^*_{n,max} + \delta$. Hence the constant $v^*_{n,max}$ is the best lower bound. 
\end{proof}

%
%

\end{section}

\section*{Acknowledgment}
To my PhD adviser Pieter Allaart for his guidance in this project.

\end{document}